%% file: main.tex
\documentclass[11pt]{article}

\usepackage[a4paper,margin=25mm]{geometry}

\usepackage[T1]{fontenc}
\usepackage{lmodern}

\usepackage{amsmath,amssymb,amsthm}
\usepackage{mathtools}

\usepackage{algorithm}
\usepackage{algpseudocode}
\algrenewcommand\algorithmicforall{\textbf{for each}}

\theoremstyle{definition}
\newtheorem{theorem}{Theorem}[section]

\newtheorem{proposition}[theorem]{Proposition}

\theoremstyle{definition}
\newtheorem{definition}[theorem]{Definition}
\newtheorem{example}[theorem]{Example}

\theoremstyle{definition}
\newtheorem{remark}[theorem]{Remark}

\input{macros.tex}
\input{papers.tex}

\begin{document}
\maketitle
\input{sections/0_abstract.tex}

\input{sections/1_introduction.tex}
\input{sections/2_preliminaries.tex}
\input{sections/3_matsumoto.tex}

\input{sections/4_main_results.tex}

\input{sections/98_Acknowledgements.tex}

\input{sections/99_bib.tex}
\end{document}

%% file: macros.tex
\title{On the Radical Computation of Parametric Ideals\\over Finite Fields}
\author{
  Kazuki Tanaka\\
  Department of Mathematical Sciences, \\
  Graduate School of Science, \\
  Tokyo Metropolitan University, \\
  1-1 Minami-Ohsawa, Hachioji 192-0397, Japan\\
  \texttt{tanaka-kazuki@ed.tmu.ac.jp}
}
\date{}

\DeclareMathOperator{\LC}{lc}
\DeclareMathOperator{\LPP}{lpp}

\DeclareMathOperator{\PREM}{prem}

\newcommand{\spn}[1]{\langle #1\rangle}
\newcommand{\set}[2]{\left\{ #1 \,\middle|\, #2 \right\}}

\newcommand{\Fq}{\mathbb{F}_q}

\newcommand{\lpp}[1]{\LPP\!\left(#1\right)}
\newcommand{\lc}[1]{\LC\!\left(#1\right)}

\newcommand{\prem}[2]{\PREM_{#1}({#2})}

%% file: papers.tex
\newcommand{\etal}{\textit{et al.\ }}
\newcommand{\symbcomp}{\textit{J. of Symbolic Computation} }
\newcommand{\issac}{\textit{Proc. of ISSAC} }
\newcommand{\casc}{\textit{Proc. of CASC} }
\newcommand{\lncs}{LNCS }

\newcommand{\gtz}{
  Gianni,\,P., Trager,\,B., Zacharias,\,G.
  (1988)
  Gr\"{o}bner Bases and Primary Decomposition of Polynomial Ideals.
  \symbcomp Vol.6, 149-167.
}

\newcommand{\ktn}{
  Kuramochi,\,R., Tanaka,\,K., Nabeshima,\,K.
  (2024)
  On the Radical of a Polynomial Ideal with Parameters.
  \casc 2024.
  \lncs Vol.14938, 193-214.
  Springer, Cham.
}
\newcommand{\m}{
  Matsumoto,\,R.
  (2001)
  Computing the Radical of an Ideal in Positive Characteristic.
  \symbcomp Vol.32, 263-271.
}

\newcommand{\w}{
  Weispfenning,\,V.
  (1992)
  Comprehensive Gr\"{o}bner Bases. 
  \symbcomp Vol.14-1, 1-29.
}

\newcommand{\gpc}{
  Gao,\,S., Platzer,\,A., Clarke,\,E.\,M.
  (2011)
  Quantifier Elimination over Finite Fields Using Gröbner Bases.
  {\lncs}Vol. 6742, 140-157. 
  Springer, Berlin Heidelberg.  

}

\newcommand{\sss}{
  Suzuki,\,A., Sato,\,Y. 
  (2006)
  A Simple Algorithm to Compute Comprehensive Gr\"{o}bner Bases. 
  \issac 2006, 326-331.
}

\newcommand{\kswjsc}{
  Kapur,\,D., Sun,\,Y., Wang,\,D.
  (2013)
  An Efficient Algorithm for Computing a Comprehensive Gr\"{o}bner System of a Parametric Polynomial System.
  \symbcomp Vol.49, 27--44.
}

\newcommand{\n}{
  Nabeshima,\,K.
  (2024)
  Generic Gr\"{o}bner basis of a parametric ideal and its application to a  Comprehensive Gr\"{o}bner System.
  \textit{AAECC}.
  Vol.35, 55-70.
}

%% file: sections/0_abstract.tex
\begin{abstract}
Matsumoto's algorithm for computing the radical of a polynomial ideal 
is generalized to the parametric setting within the framework of 
symbolic computation. 
The main tool is a comprehensive Gr\"{o}bner system over a finite field, 
also known as a parametric Gr\"{o}bner basis. 
As a result, an algorithm for computing the radical of a parametric 
ideal over a finite field is proposed.
\end{abstract}

%% file: sections/1_introduction.tex
\section{Introduction}

The radical of an ideal is a fundamental notion in algebraic geometry 
and computational algebra, as it describes the associated algebraic set 
without multiplicities. 
A celebrated result on radical computation is the algorithm of Gianni, 
Trager, and Zacharias~\cite{gtz}, which has been incorporated into major 
computer algebra systems and widely applied to various problems.

For parametric ideals over fields of characteristic zero, Kuramochi, 
Tanaka, and Nabeshima~\cite{ktn} proposed an algorithm for computing 
radicals by combining comprehensive Gr\"{o}bner systems with the 
algorithm of Gianni, Trager, and Zacharias~\cite{gtz}. 
In contrast, for non-parametric ideals over finite fields, several 
algorithms for radical computation are known, including Matsumoto's 
algorithm, which iteratively computes inverse images under the Frobenius 
map~\cite{m}. 
However, to the best of our knowledge, no algorithm has been proposed 
for the parametric case over finite fields.

In this paper, we present an algorithm for computing the radical of a 
parametric ideal over a finite field. 
The main tool of the proposed method is a comprehensive Gr\"{o}bner 
system over a finite field. 
More precisely, we generalize Matsumoto's algorithm for computing the 
radical of a polynomial ideal over a finite field to the parametric 
setting by employing comprehensive Gr\"{o}bner systems.

This paper is organized as follows. 
In Section 2, we recall basic definitions and results on comprehensive 
Gr\"{o}bner systems. 
In Section 3, we review Matsumoto's algorithm for radical computation 
over finite fields. 
In Section 4, we extend Matsumoto's algorithm to the parametric setting 
by employing comprehensive Gr\"{o}bner systems over a finite field. 
We also prove correctness and termination of the resulting algorithm.

%% file: sections/2_preliminaries.tex
\section{Comprehensive Gr\"obner systems}

\subsection{Notation}

Let $p$ be a prime number, and let $q$ be a power of $p$.
We denote by $\Fq$ the finite field of order $q$.

Let $U=\{u_1,\dots,u_m\}$ and $X=\{x_1,\dots,x_n\}$ be sets of variables with $U \cap X= \emptyset$.
We denote the parametric polynomial ring $\Fq[u_1,\dots,u_m][x_1,\dots,x_n]$ by $\Fq[U][X]$ (we often regard $U$ as parameters).
The symbol $\mathrm{Term}(X)$ means the set of terms of $X$.

Fix a term order $\prec$ on $\mathrm{Term}(X)$, and let $f \in \Fq[U][X]$.
Then $\lc{f}$, and $\lpp{f}$ denote the leading coefficient, and the leading power product of $f$, respectively.

For $a=(a_1,\dots,a_m)\in \Fq^m$, we denote by 
$
\pi_a : \Fq[U][X] \to \Fq[X]
$
the specialization map obtained by substituting $a$ for $U$.
For a set of polynomials $E\subset \Fq[U]$, we define
\[
\mathbf{V}(E)=\set{a\in \Fq^m}{g(a)=0 \text{ for all }g\in E}.
\]
For finite sets $C_1, C_2 \subset \Fq[U]$, we set
$C_1 \cdot C_2 = \{fg \mid f \in C_1,\ g \in C_2\}$.

\subsection{Comprehensive Gr\"obner Systems}

We now recall some notions related to comprehensive Gr\"obner systems in the finite field setting.

\begin{definition}[Algebraic partition]
Let $S$ be a subset of $\mathbb{F}_q^m$.
We say that non-empty subsets $S_1, \ldots, S_r \subset \mathbb{F}_q^m$ form an \textbf{algebraic partition of $S$} if the following conditions are satisfied:
\begin{enumerate}
\renewcommand{\labelenumi}{(\roman{enumi})}
\item
Each $S_i$ is a locally closed set, i.e., there exist finite sets $E_i, N_i \subset \Fq[U]$ such that
\[
S_i = \mathbf{V}(E_i)\setminus \mathbf{V}(N_i).
\]
\item
The sets $S_1,\ldots,S_r$ are a partition of $S$, i.e.,
\[
S=\bigcup_{i=1}^r S_i,
\quad \text{and} \quad
S_i\cap S_j=\emptyset \;\; (i\neq j).
\]
\end{enumerate}
\end{definition}

\begin{definition}[Comprehensive Gr\"obner system (CGS)]\label{def:cgs}
Fix a term order $\prec$ on $\text{Term}(X)$.
Let $F$ be a finite subset of $\Fq[U][X]$, and let $E, N\subset \Fq[U]$ be finite sets.
A finite set of triples
\[
\mathcal{G}=\{(E_1,N_1,G_1),\ldots,(E_r,N_r,G_r)\}
\qquad(E_i,N_i\subset \Fq[U],\; G_i\subset \Fq[U][X])
\]
is called a \textbf{comprehensive Gr\"obner system (CGS)} of $\langle F\rangle$ on $\mathbf{V}(E)\setminus \mathbf{V}(N)$ with respect to $\prec$ if the following conditions are satisfied:
\begin{enumerate}
\renewcommand{\labelenumi}{(\roman{enumi})}
\item
The sets
\[
\mathbf{V}(E_1)\setminus \mathbf{V}(N_1),\ldots,\mathbf{V}(E_r)\setminus \mathbf{V}(N_r)
\]
form an algebraic partition of $\mathbf{V}(E)\setminus \mathbf{V}(N)$.
\item
For all $a\in \mathbf{V}(E_i)\setminus \mathbf{V}(N_i)$ and $g\in G_i$, 
$\pi_a(\lc{g})\neq 0$ and the set $\pi_a(G_i)$ is a minimal Gr\"obner basis of $\spn{\pi_a(F)}$ with respect to $\prec$ in $\Fq[X]$ if not zero ideal; otherwise $G_i=\{0\}$.
\end{enumerate}
Each triple $(E_i,N_i,G_i)$ is called a \textbf{segment} of $\mathcal{G}$.
\end{definition}

\subsection{Computing CGSs over Finite Fields}

Several algorithms for computing CGSs are known~\cite{ksw, n, ss};
however, these algorithms all assume that the parameter space is a
Cartesian product of an algebraically closed field, and hence cannot
be directly applied when the parameter space is $\Fq^m$.
In this paper, we give an algorithm for computing CGSs over $\Fq^m$.
Following the approach of Kapur--Sun--Wang~\cite{ksw}, this algorithm
is constructed by replacing its emptiness check for locally closed sets
with \textsf{IsEmptyLCS} (Algorithm~\ref{alg:emptiness}), which is
valid over $\Fq^m$.
This emptiness check is discussed in detail in
Section~4.4.
We first introduce the following notion from~\cite{ksw}.

\begin{definition}[Minimal Dickson basis]\label{def:mdbasis}
  Let $G \subset \Fq[U][X]$. A subset $G_m \subset G$ is called a
  \textbf{minimal Dickson basis} of $G$ if the following conditions
  are satisfied:
  \begin{enumerate}
    \renewcommand{\labelenumi}{(\roman{enumi})}
    \item For every $g \in G$, there exists $f \in G_m$ such that
      $\lpp{f} \mid \lpp{g}$.
    \item For any two distinct $f_1, f_2 \in G_m$, neither
      $\lpp{f_1} \mid \lpp{f_2}$ nor $\lpp{f_2} \mid \lpp{f_1}$.
  \end{enumerate}
\end{definition}

The following theorem, which is a finite-field analogue of
\cite[Theorem 4.3]{ksw}, is the basis of Algorithm~\ref{alg:cgs}.

\begin{theorem}\label{thm:ksw}
Let $F$ be a finite set of $F_q[U,X]$, E finite subset of $F_q[U]$ and $G$ the reduced Gr\"obner basis of
$\langle F \cup E \rangle \subset \Fq[U \cup X]$ with respect to a block order with $X\gg U$,
let $G_r = G \cap \Fq[U]$, and let $G_m$ be a minimal
Dickson basis of $G \setminus G_r$. If
$a \in \mathbf{V}(G_r) \setminus \mathbf{V}(h)$, where
$h = \mathrm{lcm}\{\lc{g} \mid g \in G_m\}$, then $\pi_a(G_m)$ is a
minimal Gr\"obner basis of $\langle \pi_a(F) \rangle$ with respect to
$\prec$.
\end{theorem}

\begin{proof}
The proof of \cite[Theorem 4.3]{ksw} does not use the algebraic
closedness of the field in which the parameters take their values,
and hence remains valid for specializations at points of $\Fq^m$.
\end{proof}

Based on Theorem~\ref{thm:ksw}, we obtain the following algorithm for
computing CGSs over $\Fq^m$, where $\textsf{IsEmptyLCS}(E, N)$ is a
boolean-valued function that decides whether
$\mathbf{V}(E) \setminus \mathbf{V}(N)$ is empty, and
$\mathrm{lcm}(h_1, \ldots, h_s)$ denotes the least common multiple of
$h_1, \ldots, h_s \in \Fq[U]$. We set
$N \cdot \{h_1 \cdots h_{i-1}\} = N$ for $i = 1$ in line~27.

\begin{algorithm}[H]
  \caption{\sf{CGSoverFiniteField}}
  \label{alg:cgs}
  \begin{algorithmic}[1]
    \Require $E, N \subset \Fq[U]$: finite sets; $F \subset \Fq[U][X]$: a finite
      set; $\prec$: a term order on $\mathrm{Term}(X)$.
    \Ensure A CGS of $\langle F \rangle$ on $\mathbf{V}(E) \setminus \mathbf{V}(N)$
      with respect to $\prec$.
    \If{\textsf{IsEmptyLCS}$(E, N)$}
      \State \Return $\emptyset$;
    \EndIf
    \State $G \gets$ the reduced Gr\"obner basis of $\spn{F \cup E}$ with respect to  a block order with $X\gg U$;
    \If{$1 \in G$}
      \State \Return $\{(E, N, \{1\})\}$;
    \EndIf
    \State $G_r \gets G \cap \Fq[U]$;
    \If{\textbf{not} \textsf{IsEmptyLCS}$(E, G_r \cdot N)$}
      \State $\mathcal{P} \gets \{(E, G_r \cdot N, \{1\})\}$;
    \Else
      \State $\mathcal{P} \gets \emptyset$;
    \EndIf
    \If{\textsf{IsEmptyLCS}$(G_r, N)$}
      \State \Return $\mathcal{P}$;
    \EndIf
    \If{$G \setminus G_r = \emptyset$}
      \State \Return $\mathcal{P} \cup \{(G_r, N, \{0\})\}$;
    \EndIf
    \State $G_m \gets$ a minimal Dickson basis of $G \setminus G_r$;
    \State $h \gets \mathrm{lcm}\{\lc{g} \mid g \in G_m\}$;
    \State $\{h_1, \ldots, h_s\} \gets$ the distinct irreducible factors of $h$;
    \If{\textbf{not} \textsf{IsEmptyLCS}$(G_r, N \cdot \{h\})$}
      \State $\mathcal{P} \gets \mathcal{P} \cup \{(G_r, N \cdot \{h\}, G_m)\}$;
    \EndIf
    \For{$i = 1, \ldots, s$}
      \State $\mathcal{P} \gets \mathcal{P} \cup
        \textsf{CGSoverFiniteField}(G_r \cup \{h_i\},\,
        N \cdot \{h_1 \cdots h_{i-1}\},\, G \setminus G_r,\, \prec)$;
    \EndFor
    \State \Return $\mathcal{P}$;
  \end{algorithmic}
\end{algorithm}

\begin{proposition}\label{prop:cgs-correct}
Algorithm~\ref{alg:cgs} is correct and terminates in finitely many steps.
\end{proposition}

\begin{proof}
Since $E \subset \langle G_r \rangle$, the set
$\mathbf{V}(E) \setminus \mathbf{V}(N)$ is the disjoint union of
$\mathbf{V}(E) \setminus \mathbf{V}(G_r \cdot N)$,
$\mathbf{V}(G_r) \setminus \mathbf{V}(N \cdot \{h\})$, and
$\mathbf{V}(G_r \cup \{h_i\}) \setminus
\mathbf{V}(N \cdot \{h_1 \cdots h_{i-1}\})$ for $i = 1, \ldots, s$.
Since only empty sets are discarded, condition (i) of
Definition~\ref{def:cgs} holds by induction on the recursion.

Note that $\langle \pi_a(F) \rangle = \langle \pi_a(G) \rangle$ for
every $a \in \mathbf{V}(E)$. If $1 \in G$, or if
$\pi_a(g) \neq 0$ for some $g \in G_r$, then
$\langle \pi_a(F) \rangle = \langle 1 \rangle$; this justifies the
segments with $\{1\}$. If $G \setminus G_r = \emptyset$, then
$\langle \pi_a(F) \rangle = \langle \pi_a(G_r) \rangle = \{0\}$ on
$\mathbf{V}(G_r)$; this justifies the segment with $\{0\}$. For every
$a \in \mathbf{V}(G_r) \setminus \mathbf{V}(h)$, the set $\pi_a(G_m)$
is a minimal Gr\"obner basis of $\langle \pi_a(F) \rangle$ by
Theorem~\ref{thm:ksw}. The remaining segments are correct by
induction.

Finally, we show that the algorithm terminates. Whenever the algorithm
makes a recursive call, we have $h_i \notin \langle G_r \rangle$ for each $i$;
indeed, since $h_i$ divides $\lc{g}$ for some $g \in G_m$, if
$h_i \in \langle G_r \rangle$, then $\lc{g} \in \langle G_r \rangle$ as well,
and hence the leading monomial of $g$ could be reduced further by $G_r$,
which contradicts the assumption that $G$ is reduced. This shows that the
ideal $\langle E \rangle$ strictly increases at each recursive call. Since
$\Fq[U]$ is a Noetherian ring, such a strictly increasing sequence of ideals
cannot continue indefinitely, and hence the algorithm terminates after
finitely many recursive calls.
\end{proof}

\begin{example}
Consider
\[
F=
\{\, ax+y,\ x+bz,\ xz^3-az^4,\ yz^4+z^5 \,\}
\subset \mathbb{F}_3[a,b][x,y,z],
\]
with the lexicographic order such that \(x\succ y\succ z\).  A comprehensive Gr\"obner system
of \(\langle F\rangle\) is
\[
\begin{aligned}
\mathcal{G}=\{&
(\{0\},\{a+b\},
 \{x+bz,\ y-abz,\ (a+b)z^4\}),\\
&(\{a+b\},\{b^2-1\},
 \{x+bz,\ y+b^2z,\ (b^2-1)z^5\}),\\
&(\{a+b,\ b^2-1\},\{1\},
 \{x+bz,\ y+z\})\}.
\end{aligned}
\]
This set $\mathcal{G}$ means the following:
\begin{itemize}
\item If \((a,b)\) belongs to \(\mathbb{F}_3^2\setminus \mathbf{V}(a+b)\), then
\[
\{x+bz,\ y-abz,\ (a+b)z^4\}
\]
is a minimal Gr\"obner basis of \(\langle \pi_{(a,b)}(F)\rangle\).

\item If \((a,b)\) belongs to \(\mathbf{V}(a+b)\setminus \mathbf{V}(b^2-1)\), then
\[
\{x+bz,\ y+b^2z,\ (b^2-1)z^5\}
\]
is a minimal Gr\"obner basis of \(\langle \pi_{(a,b)}(F)\rangle\).

\item If \((a,b)\) belongs to \(\mathbf{V}(a+b,b^2-1)\), then
\[
\{x+bz,\ y+z\}
\]
is a minimal Gr\"obner basis of \(\langle \pi_{(a,b)}(F)\rangle\).
\end{itemize}
\end{example}

%% file: sections/3_matsumoto.tex
\section{Matsumoto's Algorithm}

Here we review an algorithm, introduced by Matsumoto in \cite{m}, for computing the radical of an ideal in the case where the coefficient field is $\Fq$.
For an ideal $I \subset \Fq[X]$, the \textbf{radical} of $I$ is defined as
\[
\sqrt{I} = \{ f \in \Fq[X] \mid f^k \in I
\text{ for some } k \in \mathbb{Z}_{>0} \}.
\]
Let
\[
\varphi : \Fq[X] \to \Fq[X], \quad
f(x_1, \ldots, x_n) \mapsto f(x_1^q, \ldots, x_n^q)
\]
be the Frobenius map.

\begin{theorem}[\cite{m}, Chapter 2]\label{thm:matsumoto-radical}
Let $I$ be an ideal of $\Fq[X]$. Then the following properties hold:
\begin{enumerate}
\renewcommand{\labelenumi}{(\roman{enumi})}
\item $I \subset \varphi^{-1}(I) \subset \sqrt{I}$.
\item If $I \ne \sqrt{I}$, then $I \subsetneq \varphi^{-1}(I)$.
\end{enumerate}
\end{theorem}

By this theorem, we obtain an ascending chain of ideals
\[
I \subset \varphi^{-1}(I) \subset \varphi^{-1}\circ\varphi^{-1}(I) \subset \cdots .
\]
Since $\Fq[X]$ is a Noetherian ring, there exists $d \in \mathbb{Z}_{\ge 0}$ such that
\[
\varphi^{-d}(I) = \varphi^{-(d+1)}(I) = \sqrt{I},
\]
where $\varphi^{-d}$ denotes the $d$-fold composition
$\underbrace{\varphi^{-1} \circ \cdots \circ \varphi^{-1}}_{d}$,
and $\varphi^{0}$ is the identity map.

Hence, the radical can be computed by the following procedure.

\begin{enumerate}
\renewcommand{\labelenumi}{(\roman{enumi})}
\item For an ideal $I$, compute $J=\varphi^{-1}(I)$.
\item Test whether $I=J$.
\item If they are equal, then $I=\sqrt{I}$. Otherwise, replace $I$ by $J$ and repeat.
\end{enumerate}

Therefore, to compute the radical, we need to compute inverse images under the Frobenius map and test whether two ideals are equal. 
The latter can be done using a reduced Gr\"obner basis. 

The following theorem is fundamental for computing inverse images under the Frobenius map.

\begin{theorem}[\cite{m}, Proposition 2.5]\label{thm:image}
Let $F$ be a finite set in $\Fq[X]$. Introduce a new set of variables $Y=(y_1,\ldots,y_n)$ and define the ideal in $\Fq[X,Y]$
\[
J = \spn{F}_{\Fq[X,Y]} + \spn{y_1-x_1^q,\ldots,y_n-x_n^q}.
\]
Then $\varphi^{-1}(\spn{F})$ is generated by the polynomials obtained from the elements of $J\cap \Fq[Y]$ by replacing each $y_i$ with $x_i$.
\end{theorem}

By computing a Gr\"obner basis of $J$ with respect to a block order with $X \gg Y$, one can obtain a generating set for $\varphi^{-1}(\langle F \rangle)$.

Consequently, the radical can be computed by the following algorithm.

\begin{algorithm}[H]
  \caption{\sf{Radical} (Matsumoto's algorithm)}
  \label{alg:radical}
  \begin{algorithmic}[1]
    \Require $F \subset \Fq[X]$: a generating set of an ideal.
    \Ensure a generating set of the radical $\sqrt{\spn{F}}$.

    \Loop
      \State $B \gets F \cup \{y_1-x_1^q,\ldots,y_n-x_n^q\} \subset \Fq[X,Y]$;
      \State $G \gets$ a Gr\"obner basis of $\spn{B}$ with respect to a block order $X \gg Y$;
      \State $H \gets$ the set obtained from $G \cap \Fq[Y]$ by substituting $x_i$ for each $y_i$;

      \If{$\spn{F}=\spn{H}$}
        \State \Return $F$;
      \Else
        \State $F \gets H$;
      \EndIf
    \EndLoop
  \end{algorithmic}
\end{algorithm}

We illustrate Algorithm~\ref{alg:radical} with the following example.
\begin{example}
Consider
$F = \{x_1^9+x_2^3,\ x_2^5-x_1x_2^3\} \subset \mathbb{F}_3[x_1,x_2]$.
To compute the inverse image of $\langle F\rangle$ under the Frobenius
map, introduce new variables $y_1, y_2$ and consider the ideal
$J = \langle F \cup \{y_1-x_1^3,\ y_2-x_2^3\}\rangle
\subset \mathbb{F}_3[x_1,x_2,y_1,y_2]$.
Computing a Gr\"obner basis of $J$ with respect to a block order
satisfying $x_1, x_2 \gg y_1, y_2$, we obtain
$J\cap \mathbb{F}_3[y_1,y_2]
= \langle y_1^3+y_2,\ y_2^3-y_1y_2\rangle$,
and hence, by replacing $y_1, y_2$ with $x_1, x_2$,
\[
  F_1 = \{x_1^3+x_2,\ x_2^3-x_1x_2\},
  \qquad
  \varphi^{-1}(\langle F\rangle)=\langle F_1\rangle.
\]
Since $\langle F\rangle \neq \langle F_1\rangle$, the algorithm
continues. Repeating the same procedure, we obtain
\[
  F_2
  =
  \{x_1^3+x_2,\ x_1x_2^2-x_1^2,\ x_2^3-x_1x_2\},
  \qquad
  \varphi^{-1}(\langle F_1\rangle)=\langle F_2\rangle,
\]
and
\[
  F_3
  =
  \{x_1^3+x_2,\ x_2^2-x_1\},
  \qquad
  \varphi^{-1}(\langle F_2\rangle)=\langle F_3\rangle.
\]
Moreover, one more application gives
$\varphi^{-1}(\langle F_3\rangle)=\langle F_3\rangle$.
Therefore, Algorithm~\ref{alg:radical} terminates and returns $F_3$,
i.e., \[\sqrt{\langle F\rangle} = \langle x_1^3+x_2,\ x_2^2-x_1\rangle.\]
\end{example}

%% file: sections/4_main_results.tex
\section{Algorithm for Computing the Radical of a Parametric Ideal}

In this section, we extend Matsumoto's algorithm to the parametric setting via comprehensive Gröbner systems, thereby presenting an algorithm for computing the radical of parametric ideals over the finite field $\Fq$.
Throughout this section, we fix a term order $\prec$ on $\mathrm{Term}(X)$.

\subsection{Parametric radical system}

Our goal is to compute the radical of a parametric ideal.
This problem is formulated in terms of a parametric radical system as follows.

\begin{definition}\label{def:prs}
Let $F$ be a finite subset of $\Fq[U][X]$, and let $E, N\subset \Fq[U]$ be finite sets.
A finite set of triples
\[
\mathcal{G}=\{(E_1,N_1,G_1),\ldots,(E_r,N_r,G_r)\}
\qquad(E_i,N_i\subset \Fq[U],\; G_i\subset \Fq[U][X])
\]
is called a \textbf{parametric radical system} of $\spn{F}$ on $\mathbf{V}(E)\setminus \mathbf{V}(N)$ if the following conditions are satisfied:
\begin{enumerate}
  \renewcommand{\labelenumi}{(\roman{enumi})}
  \item
  The sets $\mathbf{V}(E_1)\setminus \mathbf{V}(N_1),\ldots,\mathbf{V}(E_r)\setminus \mathbf{V}(N_r)$ form an algebraic partition of $\mathbf{V}(E)\setminus \mathbf{V}(N)$.
  \item
  For any $a\in \mathbf{V}(E_i)\setminus \mathbf{V}(N_i)$, the set $\pi_a(G_i)$ generates the radical ideal $\sqrt{\spn{\pi_a(F)}_{\Fq[X]}}$.
\end{enumerate}
Each triple $(E_i,N_i,G_i)$ is called a \textbf{segment} of $\mathcal{G}$.
We say that $\mathcal{G}$ is a parametric radical system of $F$ if $\mathbf{V}(E)\setminus \mathbf{V}(N)=\Fq^m$.
\end{definition}

This notion follows the definition of a parametric radical system introduced in \cite{ktn}.

As in Matsumoto's algorithm, Theorem~\ref{thm:matsumoto-radical} implies that the following procedure computes the radical of a parametric ideal:
\begin{enumerate}
  \renewcommand{\labelenumi}{(\roman{enumi})}
  \item
  Given a generating set $F \subset \Fq[U][X]$ of a parametric ideal, compute a comprehensive Gröbner system $\{(E_i, N_i, G_i)\}_i$ of $\varphi^{-1}(\spn{F})$ on $\mathbf{V}(E)\setminus \mathbf{V}(N)$.
  \item
  For each segment $(E_i, N_i, G_i)$, compute the locally closed subset $S_i \subset \mathbf{V}(E_i)\setminus \mathbf{V}(N_i)$ of the parameter space that satisfies, for all $a \in S_i$, $\langle \pi_a(G_i) \rangle =\langle \pi_a(F)\rangle$.
  \item
  If $S_i = \mathbf{V}(E_i)\setminus \mathbf{V}(N_i)$, equivalently, if
  \[
    (\mathbf{V}(E_i)\setminus \mathbf{V}(N_i))\setminus S_i = \emptyset,
  \]
  then terminate on this branch.
  Otherwise, replace $F$ by $G_i$, update $E$ and $N$ so that
  \[
    \mathbf{V}(E)\setminus \mathbf{V}(N) = (\mathbf{V}(E_i)\setminus \mathbf{V}(N_i))\setminus S_i,
  \]
  and repeat the procedure on this branch.
\end{enumerate}

Therefore, the computation of the radical reduces to the following tasks: computing the inverse image under the Frobenius map, determining the parameter regions where two parametric ideals coincide, and testing whether these regions are empty.

\subsection{Inverse Frobenius map with parameters}

Here we give a method for computing the inverse image $\varphi^{-1}(\spn{F})$ under the Frobenius map, where $F \subset \Fq[U][X]$.
By Theorem \ref{thm:image} and the elimination theorem, an algorithm for computing the inverse image can be constructed using a comprehensive Gr\"obner system of $\langle F \cup \{y_1-x_1^q,\ldots, y_n-x_n^q\}\rangle$ with respect to a block order satisfying $X\gg Y$, as follows.

\begin{algorithm}[H]
  \caption{\sf{ParaInverseFrobeniusMap}}
  \label{alg:para-inverse-frobenius-map}
  \begin{algorithmic}[1]
    \Require $(E,N,F)$: finite sets $E, N\subset \Fq[U]$ and $F\subset \Fq[U][X]$.
    \Ensure $\mathcal{H}$: a CGS of $\varphi^{-1}(\spn{F})$ on $\mathbf{V}(E)\setminus \mathbf{V}(N)$.
    \State $B \gets F \cup \{y_1-x_1^q,\ldots,y_n-x_n^q\} \subset \Fq[U][X,Y]$;
    \State $\mathcal{G} \gets$ a CGS of $\spn{B}$ on $\mathbf{V}(E)\setminus \mathbf{V}(N)$ with respect to a block order $X\gg Y$;
    \State $\mathcal{H} \gets \emptyset$;
    \ForAll{$(E_i, N_i, G_i) \in \mathcal{G}$}
      \State $\tilde{G}_i \gets G_i \cap \Fq[U][Y]$;
      \State $H_i \gets$ the set obtained from $\tilde{G}_i$ by replacing each variable $y_j$ with $x_j$;
      \State $\mathcal{H} \gets \mathcal{H} \cup \{(E_i, N_i, H_i)\}$;
    \EndFor
    \State \Return $\mathcal{H}$;
  \end{algorithmic}
\end{algorithm}

Note that, by the elimination theorem, the output of the algorithm above is a CGS of $\varphi^{-1}(\langle F \rangle)$.

\subsection{Locally closed sets for ideal equality}

Here we discuss a method for computing the locally closed set on which two parametric ideals coincide.
Let $F$ be a finite set in $\mathbb{F}_q[U][X]$ and let $(E,N, G)$ be a segment of a CGS of $\langle F \rangle$ where $G\neq \{0\}$. 
Let ${\mathcal H}$ be the output of $\mathsf{ParaInverseFrobeniusMap}(E,N,G)$ and let $(\tilde{E},\tilde{N},\tilde{G})$ be a segment of ${\mathcal H}$ where $\tilde{G}\neq \{0\}$.
Then the following four conditions hold.
\begin{enumerate}
  \renewcommand{\labelenumi}{(\arabic{enumi})}
  \item
    For every $a\in \mathbf{V}(E)\setminus \mathbf{V}(N)$, the set $\pi_a(G)$ is a Gröbner basis of $\langle \pi_a(G)\rangle$ with respect to $\prec$.
  \item
    For every $a\in \mathbf{V}(E)\setminus \mathbf{V}(N)$ and every $g\in G$, we have $\pi_a(\lc{g})\neq 0$.
  \item
    For every $a\in \mathbf{V}(\tilde{E})\setminus \mathbf{V}(\tilde{N})$, we have
    $\spn{\pi_a(G)}\subset \spn{\pi_a(\tilde{G})}$.
  \item
    The inclusion $\mathbf{V}(E)\setminus \mathbf{V}(N)\supset \mathbf{V}(\tilde{E})\setminus \mathbf{V}(\tilde{N})$ holds.
\end{enumerate}

Conditions (1) and (2) follow from the fact that $(E,N,G)$ is a segment of a CGS.
Moreover, condition (3) follows from Theorem~\ref{thm:matsumoto-radical}.
Therefore, to verify the equality of ideals  $\langle \pi_a(G) \rangle= \langle \pi_a(\tilde{G}) \rangle$, 
it remains to determine whether the reverse inclusion $\langle \pi_a(G) \rangle \supset \langle \pi_a(\tilde{G}) \rangle$.

To test this inclusion, we introduce an analogue of the normal form in the parametric setting.

\begin{proposition}\label{prop:para_rem}
  Let $F$ be a finite set in $\mathbb{F}_q[U][X]$ and let $(E,N, G)$ be a segment of a CGS of $\langle F \rangle$ where $G\neq \{0\}$. 
  Let ${\mathcal H}$ be the output of $\mathsf{ParaInverseFrobeniusMap}(E,N,G)$ and let $(\tilde{E},\tilde{N},\tilde{G})$ be a segment of ${\mathcal H}$ where $\tilde{G}\neq \{0\}$.
  For $f\in \Fq[U][X]$, regard $f$ and the elements of $G$ as elements of $\Fq(U)[X]$, and let $\rho_G(f)\in \Fq(U)[X]$ be the remainder obtained by dividing $f$ by $G$.
  Then the following statements hold:
  \begin{enumerate}
    \renewcommand{\labelenumi}{(\roman{enumi})}
    \item
      There exists $c_f\in \Fq[U]\setminus\{0\}$ such that
      $c_f\,\rho_G(f)\in \Fq[U][X]$ and
      $\pi_a(c_f)\neq 0$ for every $a \in \mathbf{V}(E)\setminus \mathbf{V}(N)$.
    \item
      Fix one such $c_f$ as in~(i),
      then for every $a \in \mathbf{V}(E)\setminus \mathbf{V}(N)$, the condition
      $\pi_a(c_f\,\rho_G(f))=0$
      is equivalent to
      $\pi_a(f)\in \spn{\pi_a(G)}$.
  \end{enumerate}
\end{proposition}

\begin{proof}
  Let $G=\{g_1,\dots,g_s\}$.
  We first prove~(i).
  By the division algorithm, only $\lc{g_1},\ldots,\lc{g_s}$ can appear in the denominators of the coefficients of the quotients and the remainder obtained by dividing $f$ by $G$.
  Hence there exist nonnegative integers $d_1,\ldots,d_s$ such that
  \[
    \lc{g_1}^{d_1}\cdots\lc{g_s}^{d_s}\rho_G(f)\in \Fq[U][X].
  \]
  Set
  \[
    c_f=\lc{g_1}^{d_1}\cdots\lc{g_s}^{d_s}\in \Fq[U]\setminus\{0\}.
  \]
  Then, by the definition of a CGS, for every $a \in \mathbf{V}(E)\setminus \mathbf{V}(N)$ we have
  $\pi_a(c_f)\neq 0$.

  Next, we prove~(ii).
  By the definition of $\rho_G(f)$, there exist $q_1,\dots,q_s\in \Fq(U)[X]$ such that
  \[
    f=\sum_{i=1}^s q_i g_i+\rho_G(f).
  \]
  Moreover, no term of $\rho_G(f)$ is divisible by any of
  $\lpp{g_1},\ldots,\lpp{g_s}$.
  Take any $a \in \mathbf{V}(E)\setminus \mathbf{V}(N)$.
  Then
  \[
    \pi_a(f)=\sum_{i=1}^s \pi_a(q_i)\pi_a(g_i)+\pi_a(\rho_G(f)).
  \]
  Since $\lpp{\pi_a(g_i)}=\lpp{g_i}$, no term of $\pi_a(\rho_G(f))$ is divisible by any of
  $\lpp{\pi_a(g_1)},\ldots,\lpp{\pi_a(g_s)}$.
  Therefore, $\pi_a(\rho_G(f))$ is the remainder obtained by dividing $\pi_a(f)$ by $\pi_a(G)$.
  Since $\pi_a(G)$ is a Gröbner basis of $\spn{\pi_a(G)}$, $\pi_a(\rho_G(f))$ coincides with the remainder obtained by dividing $\pi_a(f)$ by $\pi_a(G)$.
  Hence,
  \[
    \pi_a(c_f\,\rho_G(f))=\pi_a(c_f)\,\pi_a(\rho_G(f))=0
  \]
  if and only if
  $\pi_a(f)\in \spn{\pi_a(G)}$.
\end{proof}

\begin{definition}
  With the same notation as in Proposition~\ref{prop:para_rem}, let $f \in \mathbb{F}_q[U][X]$ and fix $c_f \in \mathbb{F}_q[U]\setminus \{0\}$ satisfying (i). We then define
  \[
    \prem{G}{f}:=c_f\,\rho_G(f).
  \]
\end{definition}

Based on the proposition above, we obtain the following criterion for computing the locally closed set on which two parametric ideals coincide.

\begin{proposition}\label{prop:ideal_equality} 
  Let $F$ be a finite set in $\mathbb{F}_q[U][X]$ and let $(E,N, G)$ be a segment of a CGS of $\langle F \rangle$ where $G\neq \{0\}$. 
  Let ${\mathcal H}$ be the output of $\mathsf{ParaInverseFrobeniusMap}(E,N,G)$ and let $(\tilde{E},\tilde{N},\tilde{G})$ be a segment of ${\mathcal H}$ where $\tilde{G}\neq \{0\}$.
  Write
  \[
    \tilde{G}=\{g_1,\dots,g_r\}, 
    \qquad
    r_i:=\prem{G}{g_i}.
  \]
  Let $H_i$ be the set of all coefficients of $r_i$, and put
  \[
    H:=\bigcup_i H_i.
  \]
  Then the following statements hold.
  
  \begin{enumerate}
    \renewcommand{\labelenumi}{(\roman{enumi})}
    \item
      For every
      $a\in\bigl(\mathbf{V}(\tilde E)\setminus \mathbf{V}(\tilde N)\bigr)\cap \mathbf{V}(H)$,
      we have
      \[
        \spn{\pi_a(G)}=\spn{\pi_a(\tilde G)}.
      \]
    \item
      For every
      $a\in (\mathbf{V}(\tilde{E})\setminus \mathbf{V}(\tilde{N}))\setminus \mathbf{V}(H)$,
      we have
      \[
        \spn{\pi_a(G)}\neq \spn{\pi_a(\tilde G)}.
      \]
  \end{enumerate}
\end{proposition}
\begin{proof}
  We first prove~(i).
  Take any
  $a\in\bigl(\mathbf{V}(\tilde E)\setminus \mathbf{V}(\tilde N)\bigr)\cap \mathbf{V}(H)$.
  Since $(\tilde E,\tilde N,\tilde G)$ is a segment of the output of
  $\mathsf{ParaInverseFrobeniusMap}(E,N,G)$, we have
  $\spn{\pi_a(G)} \subset \spn{\pi_a(\tilde G)}$.
  On the other hand, since $a\in \mathbf{V}(H)$, we have $\pi_a(r_i)=0$
  for each $i=1,\dots,r$, and hence
  $\pi_a(g_i)\in \spn{\pi_a(G)}$ by Proposition~\ref{prop:para_rem}.
  Therefore,
  \[
    \spn{\pi_a(\tilde G)} \subset \spn{\pi_a(G)},
  \]
  and the equality $\spn{\pi_a(G)}=\spn{\pi_a(\tilde G)}$ holds.

  Next, we prove~(ii).
  Take any
  $a\in \bigl(\mathbf{V}(\tilde{E})\setminus \mathbf{V}(\tilde{N})\bigr)\setminus \mathbf{V}(H)$.
  Since $a\notin \mathbf{V}(H)$, there exists some $i$ such that
  $\pi_a(r_i)\neq 0$. By Proposition~\ref{prop:para_rem}, we obtain
  $\pi_a(g_i)\notin \spn{\pi_a(G)}$, and hence
  $\spn{\pi_a(G)}\neq \spn{\pi_a(\tilde G)}$.
\end{proof}

The proposition above yields an algorithm for computing the parameter region on which two parametric ideals coincide.
Note that, we have
\[
  \bigl(\mathbf{V}(\tilde E)\setminus \mathbf{V}(\tilde N)\bigr)\cap \mathbf{V}(H)
  =\mathbf{V}(\tilde{E}\cup H)\setminus \mathbf{V}(\tilde{N}),
\]
\[
  \bigl(\mathbf{V}(\tilde{E})\setminus \mathbf{V}(\tilde{N})\bigr)\setminus \mathbf{V}(H)
  =\mathbf{V}(\tilde{E})\setminus \mathbf{V}(\tilde{N}\cdot H).
\]

\begin{algorithm}[H]
  \caption{\sf{ParaIdealEqualityLCS}}
  \label{alg:ideal-equality}
  \begin{algorithmic}[1]
    \Require $(E,N,G)$: a segment of a CGS, \quad 
    $(\tilde{E},\tilde{N},\tilde{G})$: a segment of the system of inverse Frobenius map of $G$ on $\mathbf{V}(E)\setminus \mathbf{V}(N)$.
    \Ensure $(E_{\mathcal{E}},N_{\mathcal{E}},G_{\mathcal{E}}), (E_{\mathcal{N}},N_{\mathcal{N}},G_{\mathcal{N}})$: segments satisfying
      \[
      \begin{aligned}
      \mathbf{V}(\tilde{E})&\setminus \mathbf{V}(\tilde{N})
      =
      \bigl(\mathbf{V}(E_{\mathcal{E}})\setminus \mathbf{V}(N_{\mathcal{E}})\bigr)
      \;\sqcup\;
      \bigl(\mathbf{V}(E_{\mathcal{N}})\setminus \mathbf{V}(N_{\mathcal{N}})\bigr),\\
      &\forall a\in \mathbf{V}(E_{\mathcal{E}})\setminus \mathbf{V}(N_{\mathcal{E}}),\ 
      \spn{\pi_a(G)}=\spn{\pi_a(\tilde{G})},\\
      &\forall a\in \mathbf{V}(E_{\mathcal{N}})\setminus \mathbf{V}(N_{\mathcal{N}}),\ 
      \spn{\pi_a(G)}\neq \spn{\pi_a(\tilde{G})}.
      \end{aligned}
      \]
    \State $\tilde{G} \gets \{g_1,\dots,g_r\}$;
    \State $H \gets \emptyset$;
    \ForAll{$g_i\in\tilde{G}$}
      \State $r_i \gets \prem{G}{g_i}$;
      \State $H \gets H \cup \{\text{all coefficients of } r_i\}$;
    \EndFor
    \State $\mathcal{E} \gets (\tilde{E}\cup H,\tilde{N},G)$;
    \State $\mathcal{N} \gets (\tilde{E},\tilde{N}\cdot H,\tilde{G})$;
    \State \Return $\mathcal{E}, \mathcal{N}$;
  \end{algorithmic}
\end{algorithm}


\subsection{Emptiness of locally closed sets in $\Fq^m$}

Next, we present an algorithm for determining whether a given locally closed set
$\mathbf{V}(E)\setminus \mathbf{V}(N)$ is empty.
To this end, we use the following proposition.

\begin{proposition}\label{prop:emptiness}
  Let $E,N\subset\Fq[U]$ be finite sets, and let
  $\Phi=\{u_1^q-u_1,\dots,u_m^q-u_m\}$.
  Then the following statements are equivalent:
  \begin{enumerate}
    \renewcommand{\labelenumi}{(\roman{enumi})}
    \item
      $\mathbf{V}(E)\setminus \mathbf{V}(N)$ is empty.
    \item
      $N\subset\langle E\cup\Phi\rangle$.
  \end{enumerate}
\end{proposition}

This proposition can be proved by using the following Nullstellensatz over $\Fq$
due to Gao \etal~\cite{gpc}.

\begin{theorem}\label{thm:finite-nullstellensatz}
  \cite[Theorem 2.2]{gpc}
  Let $F\subset\Fq[U]$ be a finite set, and let
  $\Phi=\{u_1^q-u_1,\dots,u_m^q-u_m\}$.
  Then
  \[
    \mathbf{I}(\mathbf{V}(F))=\langle F\cup\Phi\rangle,
  \]
  where $\mathbf{I}(S) = \{ f \in \mathbb{F}_q[U] \mid f(a) = 0
\text{ for all } a \in S \}$ for a subset $S \subset \mathbb{F}_q^m$.
\end{theorem}

\begin{proof}[Proof of Proposition~\textnormal{\ref{prop:emptiness}}]
  We first prove (i)$\Rightarrow$(ii).
  Assume that $\mathbf{V}(E)\setminus \mathbf{V}(N)=\emptyset$.
  Then $\mathbf{V}(N)\supset \mathbf{V}(E)$, and hence
  $\mathbf{I}(\mathbf{V}(N)) \subset \mathbf{I}(\mathbf{V}(E))$.
  Therefore, by Theorem~\ref{thm:finite-nullstellensatz},
  \[
    N \subset \mathbf{I}(\mathbf{V}(N)) \subset \mathbf{I}(\mathbf{V}(E))
    = \langle E\cup\Phi\rangle.
  \]
  Next, we prove (ii)$\Rightarrow$(i).
  Assume that $N\subset \langle E\cup\Phi\rangle$.
  By Theorem~\ref{thm:finite-nullstellensatz},
  $N\subset \langle E\cup\Phi\rangle = \mathbf{I}(\mathbf{V}(E))$,
  and hence $\mathbf{V}(N)\supset \mathbf{V}(E)$.
  Therefore, $\mathbf{V}(E)\setminus \mathbf{V}(N)=\emptyset$.
\end{proof}

Therefore, we obtain the following algorithm.

\begin{algorithm}[H]
  \caption{\sf{IsEmptyLCS}}
  \label{alg:emptiness}
  \begin{algorithmic}[1]
    \Require $E,N\subset \Fq[U]$: finite sets.
    \Ensure whether $\mathbf{V}(E)\setminus \mathbf{V}(N)=\emptyset$.

    \State $\Phi \gets \{u_1^q-u_1,\dots,u_m^q-u_m\}$;

    \If{$N=\emptyset$}
      \State \Return $\mathrm{TRUE}$;
    \EndIf
    \ForAll{$p\in N$}
      \If{$p \not\in \langle E\cup\Phi\rangle$}
        \State \Return $\mathrm{FALSE}$;
      \EndIf
    \EndFor
    \State \Return $\mathrm{TRUE}$;
  \end{algorithmic}
\end{algorithm}

\begin{remark}
In the existing algorithms for computing CGSs (cf.~\cite{ksw}), the
parameter space is the Cartesian product of an algebraically closed
field, and hence the emptiness of a locally closed set is checked by
the usual Nullstellensatz. In our setting, the parameter space is
$\mathbb{F}_q^m$, and hence Algorithm~\ref{alg:emptiness} employs the Nullstellensatz
over $\mathbb{F}_q$ (Theorem~\ref{thm:finite-nullstellensatz}).
\end{remark}

\subsection{Main algorithm}

By combining the results above, we can construct an algorithm for computing the radical of a parametric ideal.
Since the ideal equality algorithm $\mathsf{ParaIdealEqualityLCS}$ requires a segment of a CGS as input, we first compute a CGS of $\spn{F}$ as a preprocessing step.
Then, for each resulting segment $(E_i,N_i,G_i)$, we apply the main routine $\mathsf{ParaRadicalMain}$.
By Algorithm~\ref{alg:para-inverse-frobenius-map}, at each step we obtain a CGS of the inverse image under the Frobenius map, and by Proposition~\ref{prop:ideal_equality}, each segment is correctly divided into the part where the ideal coincides with the original one and the part where it does not.
On the region where the two ideals coincide, Matsumoto's theorem implies that the corresponding specialized ideal is already radical. 
In contrast, on the complementary region, we update to a strictly larger ideal and repeat the same procedure.
Therefore, this algorithm yields a parametric radical system of $\spn{F}$.

\begin{algorithm}[H]
  \caption{\sf{ParaRadical}}
  \label{alg:para-radical}
  \begin{algorithmic}[1]
    \Require $F \subset \Fq[U][X]$: a generating set of a parametric ideal.
    \Ensure $\mathcal{R}$: a parametric radical system of $\spn{F}$.
    \State $\mathcal{R} \gets \emptyset$; $\mathcal{G} \gets$ a CGS of $\spn{F}$;
    \ForAll{$(E_i, N_i, G_i) \in \mathcal{G}$}
      \If{$G_i = \{0\}$ or $G_i = \{1\}$}
        \State $\mathcal{R} \gets \mathcal{R} \cup \{(E_i, N_i, G_i)\}$;
      \Else
        \State $\mathcal{R} \gets \mathcal{R} \cup \mathsf{ParaRadicalMain}(E_i, N_i, G_i)$;
      \EndIf
    \EndFor
    \State \Return $\mathcal{R}$;
  \end{algorithmic}
\end{algorithm}

\begin{algorithm}[H]
  \caption{\sf{ParaRadicalMain}}
  \label{alg:para-radical-main}
  \begin{algorithmic}[1]
    \Require $(E, N, G)$: a segment of a CGS.
    \Ensure $\mathcal{R}$: a parametric radical system of $\spn{G}$ on $\mathbf{V}(E)\setminus \mathbf{V}(N)$.
    \State\label{line:para-radical-cgs} $\mathcal{R} \gets \emptyset$;
    \State\label{line:prm-frobenius} $\mathcal{H} \gets \mathsf{ParaInverseFrobeniusMap}(E, N, G)$;
    \ForAll{$(E_i, N_i, G_i) \in \mathcal{H}$}
      \State $((E_\mathcal{E},N_\mathcal{E},G_\mathcal{E}), (E_\mathcal{N},N_\mathcal{N},G_\mathcal{N})) \gets \mathsf{ParaIdealEqualityLCS}((E,N,G), (E_i,N_i,G_i))$;
      \If{not $\mathsf{IsEmptyLCS}(E_\mathcal{E},N_\mathcal{E})$}
        \State $\mathcal{R} \gets \mathcal{R}\cup \{(E_\mathcal{E},N_\mathcal{E},G_\mathcal{E})\}$;
      \EndIf
      \If{not $\mathsf{IsEmptyLCS}(E_\mathcal{N},N_\mathcal{N})$}
        \State $\mathcal{R} \gets \mathcal{R} \cup \mathsf{ParaRadicalMain}(E_\mathcal{N}, N_\mathcal{N}, G_\mathcal{N})$;
      \EndIf
    \EndFor
    \State \Return $\mathcal{R}$;
  \end{algorithmic}
\end{algorithm}

\begin{proposition}\label{prop:para-radical-correct}
  Algorithm~\ref{alg:para-radical} is correct and terminates in finitely many steps.
\end{proposition}

\begin{proof}
  Since the ideals $\langle 0 \rangle$ and $\langle 1 \rangle$ are radical,
  it suffices to show that, for each segment $(E, N, G)$ of the CGS
  computed in line~\ref{line:para-radical-cgs} of Algorithm~\ref{alg:para-radical}
  with $G \neq \{0\}, \{1\}$, \textsf{ParaRadicalMain}
  outputs a parametric radical system of $\langle G \rangle$ on
  $\mathbf{V}(E) \setminus \mathbf{V}(N)$ and terminates in finitely
  many steps.
  
We first verify condition (i) of Definition~\ref{def:prs}. By
  Definition~\ref{def:cgs}, the locally closed sets associated with the
  segments of the CGS $\mathcal{H}$ computed in
  line~\ref{line:prm-frobenius} of \textsf{ParaRadicalMain}
  (Algorithm~\ref{alg:para-radical-main}) form an algebraic partition of $\mathbf{V}(E) \setminus \mathbf{V}(N)$, and
  \textsf{ParaIdealEqualityLCS} divides each of them into two disjoint
  locally closed sets. Since only empty sets are discarded, condition (i)
  holds for the output $\mathcal{R}$.
  
  Next, we verify condition (ii). Since $\mathcal{H}$ is the output of
  \textsf{ParaInverseFrobeniusMap}, it is a CGS of
  $\varphi^{-1}(\langle G \rangle)$ on
  $\mathbf{V}(E) \setminus \mathbf{V}(N)$; in particular, for each
  segment $(E_i, N_i, G_i) \in \mathcal{H}$ and each
  $a \in \mathbf{V}(E_i) \setminus \mathbf{V}(N_i)$, we have
  $\langle \pi_a(G_i) \rangle = \varphi^{-1}(\langle \pi_a(G) \rangle)$.
  Hence, by Theorem~\ref{thm:matsumoto-radical}~(i),
  \[
  \langle \pi_a(G) \rangle
  \subset \langle \pi_a(G_i) \rangle
  \subset \sqrt{\langle \pi_a(G) \rangle}.
  \]
  If $a \in \mathbf{V}(E_{\mathcal{E}}) \setminus
  \mathbf{V}(N_{\mathcal{E}})$, then
  $\langle \pi_a(G) \rangle = \sqrt{\langle \pi_a(G) \rangle}$
  by Proposition~\ref{prop:ideal_equality}~(i) and
  Theorem~\ref{thm:matsumoto-radical}~(ii). Since
  $G_{\mathcal{E}} = G$, the set $\pi_a(G_{\mathcal{E}})$ generates the
  radical $\sqrt{\langle \pi_a(G) \rangle}$.
  If $a \in \mathbf{V}(E_{\mathcal{N}}) \setminus
  \mathbf{V}(N_{\mathcal{N}})$, then, by applying the same argument
  repeatedly to the segment
  $(E_{\mathcal{N}}, N_{\mathcal{N}}, G_{\mathcal{N}})$, every segment
  added to $\mathcal{R}$ in the recursive calls also satisfies
  condition (ii).
  
  Finally, we prove the termination. Since
  line~\ref{line:para-radical-cgs} of Algorithm~\ref{alg:para-radical}
  and the loop over the segments involve only finitely many steps, it suffices to
  show that \textsf{ParaRadicalMain} terminates. 
  
  Assume that \textsf{ParaRadicalMain} does not terminate.
  Then recursive calls occur infinitely many times.
  Each recursive call is performed on some locally closed set
  \[
    \mathbf{V}(E_\mathcal{N})\setminus \mathbf{V}(N_\mathcal{N}) \subset \Fq^m.
  \]
  Since $\Fq^m$ is a finite set, there exists some $a\in\Fq^m$ that belongs to infinitely many of these locally closed sets.

  Let $G_0,G_1,\dots$ be the sequence of bases appearing along the recursive branch containing this $a$.
  By Proposition~\ref{prop:ideal_equality}, we have
  \[
    \langle \pi_a(G_0)\rangle
    \subsetneq
    \langle \pi_a(G_1)\rangle
    \subsetneq
    \langle \pi_a(G_2)\rangle
    \subsetneq \cdots.
  \]
  Thus, we obtain a strictly ascending chain of ideals in $\Fq[X]$, which contradicts the Noetherian property of $\Fq[X]$.
  Therefore, Algorithm~\ref{alg:para-radical} terminates.
\end{proof}

\begin{example}

Consider
\[
F=\{x^9+ax^5y+by^3,\ y^9+cy^3\}
\subset \mathbb{F}_3[a,b,c][x,y]
\]
where $a, b, c$ are parameters, with the lexicographic order such that $x \succ y$.
We apply Algorithm~\ref{alg:para-radical} to \(F\).
Since \(\{(\{0\},\{1\},F)\}\) is already a CGS of \(\langle F\rangle\) on
\(\mathbf{V}(0)=\mathbb{F}_3^3\), the radical computation starts from this segment.

Computing the inverse image under the Frobenius map for this segment, we obtain
the following four segments:
\[
\begin{aligned}
S_1={}&
(\{0\},\{abc\},
 \{x^9+ax^5y-bcy,\ x^6y^2+cx^6,\ y^3+cy\}),\\
S_2={}&
(\{a\},\{1\},
 \{x^3+by,\ y^3+cy\}),\\
S_3={}&
(\{b\},\{a\},
 \{x^6+ax^2y,\ y^3+cy\}),\\
S_4={}&
(\{c\},\{ab\},
 \{x^9+ax^5y,\ x^8y+ax^4y^2,\ x^6y^2,\ y^3\}).
\end{aligned}
\]

For \(S_1\), the coincidence region, i.e., the subset on which the ideal
generated by the current basis coincides with its inverse image under the
Frobenius map, is empty.  Hence the algorithm continues recursively.
In the first recursive call, computing the inverse image under the Frobenius map
for \(S_1\) gives
\[
S_{11}=
(\{0\},\{abc\},
 \{x^9+ax^5y-bcy,\ x^2y^2+cx^2,\ y^3+cy\}).
\]
For \(S_{11}\), the coincidence region is again empty.  Thus the algorithm
continues recursively once more.  Computing the inverse image under the
Frobenius map for \(S_{11}\) gives
\[
S_{111}=
(\{0\},\{abc\},
 \{x^9+ax^5y-bcy,\ xy^2+cx,\ y^3+cy\}).
\]
For \(S_{111}\), the coincidence region is the whole locally closed set
associated with the segment.  Therefore the recursion stops, and \(S_{111}\)
is added to the output.

For \(S_2\), the first recursive call computes the inverse image under the
Frobenius map and gives the following three segments:
\[
\begin{aligned}
S_{21}={}&
(\{a\},\{c,bc\},
 \{y^3+cy,\ cx-by\}),\\
S_{22}={}&
(\{a,c\},\{b\},
 \{x^3,\ y\}),\\
S_{23}={}&
(\{a,b,c\},\{1\},
 \{x,\ y\}).
\end{aligned}
\]
For \(S_{21}\), the coincidence region is empty.  Hence the algorithm
continues recursively.  In the next recursive call, computing the inverse image
under the Frobenius map for \(S_{21}\) gives
\[
S_{211}=
(\{a,bc^2-b\},\{c,b,bc\},
 \{y^3+cy,\ cx-by\}).
\]
For \(S_{211}\), the coincidence region is the whole locally closed set
associated with the segment.  Therefore the recursion stops, and \(S_{211}\)
is added to the output.

For \(S_{22}\), computing the inverse image under the Frobenius map gives
\[
S_{221}=
(\{a,c\},\{b\},\{x,y\}).
\]
For \(S_{221}\), the coincidence region is the whole locally closed set
associated with the segment.  Hence the recursion stops, and \(S_{221}\) is
added to the output.

For \(S_{23}\), the coincidence region is the whole locally closed set
associated with the segment.  Hence \(S_{23}\) is added to the output.

The remaining segments \(S_3\) and \(S_4\) are treated in the same way.  For
\(S_3\), the recursive calls give the following final segments:
\[
S_{311}=
(\{b\},\{ac\},
 \{x^5+axy,\ xy^2+cx,\ y^3+cy\}),\quad
S_{3211}=
(\{b,c\},\{a\},
 \{x,y\}).
\]
For \(S_4\), the recursive calls give
\[
S_{4111}=
(\{c\},\{ab\},
 \{x,y\}).
\]

Consequently,
we obtain the following parametric radical system:
\[
\mathcal{R}
=
\{S_{111},S_{211},S_{221},S_{23},S_{311},S_{3211},S_{4111}\}.
\]
\end{example}

%% file: sections/98_Acknowledgements.tex
\section*{Acknowledgements}
The author would like to express his sincere gratitude to Professor
Shunichi Yokoyama for his continuous guidance and valuable advice.
The author is also deeply grateful to Professor Katsusuke Nabeshima for
his careful reading of earlier drafts of this manuscript and for
numerous valuable discussions and suggestions that greatly improved the
paper.
The author is also grateful to Professor Kazuhiro Yokoyama for
introducing Matsumoto's paper, which served as an important starting
point for this research.